\documentclass[a4paper,12pt,oneside]{amsart}

\usepackage[german,english]{babel}

\usepackage{amsmath,amssymb,amsthm,enumerate,mathtools}
\usepackage{mathrsfs} 

\usepackage{dsfont}
\usepackage{color}
\usepackage[utf8]{inputenc} 
\usepackage[automark,nouppercase]{scrpage2}
\usepackage{hyperref} 
\usepackage[numbers]{natbib}

			\providecommand{\Prob}[1]{\mathbb{P}\{#1\}}

			\providecommand{\card}[1]{\texttt{\#}#1}

			\newcommand{\Ewens}[1]{\text{Ewens}(#1)}

			\providecommand{\risfac}[2]{#1^{\overline{#2}}} %rising factorial power
\newcommand{\mE}{\mathbb{E}}
\newcommand{\mN}{\mathbb{N}}
\newcommand{\mP}{\mathbb{P}}

\newcommand{\cQ}{\mathcal{Q}}

\newcommand{\cB}{\mathcal{B}}

\newcommand{\eps}{\varepsilon}
\newcommand{\1}{\mathds{1}}

\newtheorem {definition}{Definition}[section]
\newtheorem {theorem}[definition]{Theorem}
\newtheorem{proposition}[definition]{Proposition}
\newtheorem {lemma}[definition]{Lemma}
\newtheorem {corollary}[definition]{Corollary}

\newtheorem {remark}[definition]{Remark}

\DeclareMathOperator*{\Var}{Var}

\author{Helmut H.~Pitters}
\address{Helmut H.~Pitters: Mathematical Institute, University of Mannheim}
\email{hpitters@mail.uni-mannheim.de}
\author{Philip~Weissmann}
\address{PhilipWeissmann: Mathematical Institute, University of Mannheim}
\email{hweissma@mail.uni-mannheim.de}
\thanks{The authors acknowledge financial support by the DFG RTG 1953.}

\title[P\MakeLowercase{oisson limit for the number of cycles in random permutations}]{Poisson limit for the number of cycles in a random permutation and the number of segregating sites}

\allowdisplaybreaks

\begin{document}

\begin{abstract}
  Consider a random permutation of $\{1, \ldots, \lfloor n^{t_2}\rfloor\}$ drawn according to the Ewens measure with parameter $t_1$ and let $K(n, t)$ denote the number of its cycles, where $t\equiv (t_1, t_2)\in\mathbb [0, 1]^2$.
  
  Next, consider a sample drawn from a large, neutral population of haploid individuals subject to mutation under the infinitely many sites model of Kimura whose genealogy is governed by Kingman's coalescent. Let $S(n, t)$ count the number of segregating sites in a sample of size $\lfloor n^{t_2}\rfloor$ when mutations arrive at rate $t_1/2$.
  
  We show that $K(n, (t_1/\log n, t_2))-1$ and $S(n, (t_1/\log n, t_2))$ induce unique random measures $\Pi_n^K$ and $\Pi_n^S,$ respectively, on the positive quadrant $[0, \infty)^2.$ Our main result is to show that in the coupling of $S(n, t)$ and $K(n, t)$ introduced in~\cite{Pitters2019} we have weak convergence as $n\to\infty$
  \begin{align*}
    (\Pi_n^K, \Pi_n^S)\to_d (\Pi, \Pi),
  \end{align*}
where $\Pi$ is a Poisson point process on $[0, \infty)^2$ of unit intensity. This complements the work in~\cite{Pitters2019} where it was shown that the process $\{(K(n, t), S(n, t)), t\in [0, 1]^2\},$ appropriately rescaled, converges weakly to the product of the same one-dimensional Brownian sheet.
\end{abstract}

\maketitle 

	

\section{Introduction and main results}\label{sec_intro}
%Let $S(n,t), t= (t_1,t_2) \in [0,1]^2$, be the number of segregating sites in Kingman's coalescent tree of size $\lfloor n^{t_2} \rfloor$ with mutations arriving with rate $t_1/2$. Moreover, we denote by $K(n,t), t= (t_1,t_2) \in [0,1]^2$, the number of cycles in a random permutation which follows the Ewens$(\lfloor n^{t_2} \rfloor,t_1)$ distribution. In Pitters \cite{Pitters2019} the author showed functional central limit theorems for both of these processes but also provided a coupling of them and a central limit theorem for  $\{ (K(n,s),S(n,t)),s,t \in [0,1]^2 \}$. The exact statement is that the process
%$$\Big\{ \frac{(K(n,s),S(n,t))- (s_1 s_2,t_1 t_2) \log n}{\sqrt{\log n}}, s,t \in [0,1]^2 \Big\}$$
%converges weakly to $\{(B(s),B(t)), s,t \in [0,1]^2\}$ (for $n\rightarrow \infty$) on the space $D_4$ in the S-topology where $B$ is a one-dimensional Brownian sheet.
%
%\smallskip
%
%Here, we change the perspective and consider $\tilde{S}(n,t) := S(n,(t_1/{\log n},t_2))$ and $\tilde{K}(n,t) := K(n,(t_1/{\log n},t_2))$, where $S$ and $K$ are constructed according to the coupling mentioned above (Theorem 3 of \cite{Pitters2019}). We will show that both expressions can be interpreted as the evaluation of $(0,t] := (0,t_1] \times (0,t_2]$ under random measures $\Pi^S_n$ and $\Pi^K_n$ which both converge weakly to a Poisson point process $\Pi$ with intensity measure $\lambda^2$ for $n \rightarrow \infty$, where $\lambda^2$ denotes two-dimensional Lebesgue measure.
Let us first present the two random models that we study.

\subsection{Number of cycles in a random permutation}
For a natural number $n\in\mathbb N\coloneqq \{1, 2, \ldots \}$ let $\mathfrak S_n$ denote the symmetric group of permutations of $[n]\coloneqq \{1, \ldots, n\}.$ For any permutation $\sigma\in\mathfrak S_n$ let $\card\sigma$ denote the number of cycles in $\sigma.$ Random permutations and their cycle structure have been studied extensively and have a long history. One of the most celebrated families of probability measures on $\mathfrak S_n$ is the so-called Ewens measure parameterised by some parameter $t_1> 0.$ We say that $\Sigma(n)\equiv\Sigma(n, t_1)$ is governed by the Ewens($n, t_1$) distribution on $\mathfrak S_n$ if for any $\sigma\in\mathfrak S_n$
\begin{align}\label{eq:def_ewens}
  \Prob{\Sigma(n)=\sigma} &= \frac{t_1^{\card\sigma}}{\risfac{t_1}{n}}\mathds 1\{\sigma\in\mathfrak S_n\},
\end{align}
where for any $x\in\mathbb R,$ $\risfac{x}{n}\coloneqq x(x+1)\cdots (x+n-1)$ denotes the $n$th rising factorial power of $x$.
In this case we write $\Sigma(n)\sim\Ewens{n, t_1}$. Here for any event $A$, $\mathds 1A$ denotes its indicator which equals one if $A$ occurs and zero otherwise.
\begin{remark}
  Our notation differs from the notation in the literature where the parameter $t_1$ is usually denoted by $\theta.$
\end{remark}

In what follows we focus on $K(n, t)$, the number of cycles in the random permutation $\Sigma(\lfloor n^{t_2}\rfloor, t_1)$ for $t\equiv (t_1, t_2)\in [0, \infty)^2$. More specifically, we are interested in the asymptotic behaviour of $K(n, t)$ for large $n$. For a law of large numbers and (functional) central limit theorems for $K(n, t)$, and an overview of related results in the literature the reader is referred to~\cite{Pitters2019}. There the author provides a coupling of $K(n, t)$ in both $n$ and $t$. In particular, in this coupling the author shows weak convergence of processes as $n\to\infty$
  \begin{align}\label{eq:brownian_sheet_fdd}
    \left\{\frac{K(n, t)-t_1t_1\log n}{\sqrt{\log n}}, t\in\mathbb [0, 1]^2\right\} \to \{\mathscr B(t), t \in\mathbb [0, 1]^2\}
  \end{align}
where $\mathscr B$ denotes the one-dimensional Brownian sheet on the unit square. Here we work in the same coupling as mentioned above. However, as opposed to~\cite{Pitters2019}, here we are interested in the large $n$ limit of $K(n, t)$ when the parameter $t_1$ is replaced by $t_1/\log n$. We will see that $t\mapsto K(n, (t_1/\log n, t_2))-1$ interpreted as a distribution function induces a point process $\Pi_n^K$ on $[0, \infty)^2$ whose asymptotics we study.
%Our first result is the weak convergence of point processes
%\begin{align}
%  \Pi_n^K\to_d \Pi\qquad \text{ as }n\to\infty,
%\end{align}
%where $\Pi$ denotes a Poisson point process on $[0, \infty)^2$ of unit intensity.

\subsection{Number of segregating sites}

In large neutral populations of haploid individuals the genealogy of a sample of $n$ individuals is often modeled by Kingman's $n$-coalescent, and there are rigorous mathematical results justifying this approximation. A verbal description of this stochastic process is as follows. Picture the individuals in the sample labeled $1, \ldots, n$, with a line of descent emanating from each individual and growing at unit speed. At rate one any pair of individuals merges, i.e.~their lines of descent merge into a single line representing the most recent common ancestor of this pair. After the first merger the process continues with $n-1$ lines of descent following the same dynamics as before. It is clear from this description that the genealogy of a sample of $n$ individuals may be represented as a (random) rooted tree with $n$ leafs labeled $1, \ldots, n.$ 
%For a formal definition of Kingman's coalescent see section~\ref{sec:background}.
			
			In addition to the genealogy mutations are modeled as follows. Conditionally given the genealogical tree (or coalescent tree), throw down points onto the branches of the tree (identified with intervals of the real line) according to a Poisson point process with constant intensity $t_1/2>0$, the so-called mutation rate. Each point of the Poisson process is then interpreted as a mutation affecting any leaf (the individual in the sample) with the property that the unique path connecting the leaf to the root of the tree crosses said mutation. A formal way to define this procedure is to identify Kingman's coalescent with a random ultrametric space on which a Poisson process can then be defined. However, this is beyond the scope of this article, and we refer the interested reader to Evans' lecture notes~\cite{Evans2007} instead.
			
			We restrict ourselves to the infinitely many sites model of Kimura~\cite{Kimura1969}. According to Kimura's model each mutation is thought of as acting on one of infinitely many sites, i.e.~each jump of the Poisson process on the tree introduces a mutation on a site where no mutation was previously observed. For detailed expositions of probabilistic models for the evolution of DNA sequences the interested reader is referred to Durrett~\cite{Durrett2008}, Etheridge~\cite{Etheridge2011}, and Tavar\'{e}~\cite{Tavare2004}.
			
Let $S(n, t)$ denote the number of segregating sites in Kingman's $\lfloor n^{t_2}\rfloor$-coalescent with mutations arriving at rate $t_1/2$. Watterson~\cite{Watterson1975} showed a law of large numbers and a central limit theorem for $S(n, (t_1, 1))$. This central limit theorem was extended in~\cite{Pitters2019} to a functional central limit theorem in a coupling of $S(n, t)$ for both $n$ and $t$ induced in a natural way by Kingman's coalescent.
Namely, it was shown that as $n\to\infty$ one has weak convergence of processes
  \begin{align}\label{eq:brownian_sheet_fdd}
    \left\{\frac{S(n, t)-t_1t_2\log n}{\sqrt{\log n}}, t\in\mathbb [0, 1]^2\right\} \to \{\mathscr B(t), t \in\mathbb [0, 1]^2\},
  \end{align}
where $\mathscr B$ is a one-dimensional Brownian sheet. Again, we work in the same coupling induced by Kingman's coalescent, but we replace the mutation rate $t_1$ by $t_1/\log n$. We show that $S(n, (t_1/\log n, t_2))$ can be interpreted as a distribution function inducing a point process $\Pi_n^S$ on $[0, \infty)^2$. 

%\smallskip
%
%Our first result then is weak convergence of processes
%\begin{align}
%  \Pi_n^S\to_d \Pi\qquad\text{ as }n\to\infty,
%\end{align}
%where $\Pi$ is a Poisson point process with unit intensity on $[0, \infty)^2$. Finally, in our main result we show joint convergence
%\begin{align}
%  (\Pi_n^K, \Pi_n^S)\to_d (\Pi, \Pi)\qquad\text{ as }n\to\infty.
%\end{align}
%From this we can also obtain the 
%weak convergence of processes
%\begin{align}
%  \Pi_n^K\to_d \Pi\qquad\text{ as }n\to\infty.
%\end{align}

\smallskip

 Our main result provides the joint convergence of $\Pi_n^S$ and $\Pi_n^K$.
\begin{theorem}\label{main_result2}
The sequence of joint random measures $\{(\Pi^S_n,\Pi_n^K),n \in \mN\}$ converges weakly to $(\Pi,\Pi)$ for $n \rightarrow \infty$, where $\Pi$ is a Poisson random measure on $[0,\infty)^2$ with intensity measure $\lambda^2$.
\end{theorem}

\section{Preliminaries}
In this section we collect results from the literature that we will use in the sequel.

\textbf{The coupling.} Kingman's coalescent suggests a natural way to couple the Ewens measure on permutations together with the number of segregating sites. We recall this coupling which was introduced in~\cite{Pitters2019}.
To define the random measures $\Pi^S_n$ and $\Pi^K_n$ we fix a Poisson point process $\mathfrak{P}$ on $[0,\infty)^2$ with intensity measure $\lambda^2/2$, where $\lambda^2$ is  Lebesgue measure on $\cB([0,\infty)^2)$.

By $L_n$ let us denote the length of Kingman's coalescent tree on $n \geq 2$ leaves. It follows that
\begin{align}\label{Ln_exp}
L_n =_d \sum_{k=2}^n \tau_k,
\end{align}
where the $(\tau_k)$ are independent random variables such that $\tau_k$ obeys an exponential distribution with parameter $\binom{k}{2}$. Furthermore, let
$$\mathcal{Q}_0:= \Big\{[0,t_1) \times [0,t_2) : t=(t_1,t_2) \in [0,\infty)^2 \Big\}$$
be the set of rectangles in $[0,\infty)^2$ having a corner point in the origin.  We define two random functions on $\mathcal{Q}_0$ by setting
\begin{align*}
\Pi^S_n[0,t):=& \,\card  \big( \mathfrak{P} \cap [0,L_{\lfloor n^{t_2} \rfloor}) \times [0,t_1/{\log n}) \big) \\
=&\,\sum_{k=2}^{\lfloor n^{t_2} \rfloor} \card  \big( \mathfrak{P} \cap [L_{k-1},L_{k}) \times [0,t_1/{\log n}) \big)
\end{align*}
and
\begin{align*}
\Pi^K_n[0,t) := \sum_{k=2}^{\lfloor n^{t_2} \rfloor} \1 \Big\{ \big(\card \mathfrak{P} \cap [L_{k-1},L_{k}) \times [0,t_1/{\log n}) \big) \geq 1 \Big\}
\end{align*}
for $t = (t_1,t_2) \in [0,\infty)^2$. Using the coupling of $S$ and $K$ one can see that $\Pi^S_n[0,t) =_d \tilde{S}(n,t):=S(n,(t_1/ \log n, t_2))$ and $\Pi^K_n[0,t) + 1 =_d \tilde{K}(n,t):= K(n,(t_1/ \log n, t_2))$.

Our first result states that $\Pi_n^S$ and $\Pi_n^K$ can be extended to random measures.
\begin{proposition}\label{prop_measures}
The random functions $\Pi_n^S$ and $\Pi_n^K$ on $\cQ_0$ can be extended to random measures on $\cB([0,\infty)^2)$.
\end{proposition}
This result provides the interpretation of $t\mapsto\tilde{S}(n,t)$ and $t\mapsto\tilde{K}(n,t)$ as the distribution functions of random measures $\Pi_n^S$ and $\Pi_n^K$, respectively.

We collect some necessary conditions for convergence of point processes that we need in the sequel. Before, we need the following conventions: First, we call a set $\mathcal{U}$ of subsets of a state space $S$ separating if for all compact $C \subseteq S$ and open $G \subseteq S$ with $C \subseteq G$ there is an $U \in \mathcal{U}$ such that $C \subseteq U \subseteq G$. Moreover, a point process $\xi$ on a state space $S$ is called simple if $\xi \{ s \} \leq 1$ for all singletons $s \in S$.

\begin{proposition}[Proposition 16.17 of \cite{MR2161313}] \label{prop_Kallenberg} 
Let $\{ \xi_n : n \in \mN \}$ be a sequence of point processes and let $\xi$ be a simple point process on a locally compact, second countable Hausdorff space $S$. Then $\xi_n \to_d \xi$ as $n \rightarrow \infty$ under the following conditions:
\begin{enumerate}[(i)]
\item $\lim_{n \rightarrow \infty} \mP(\xi_n U =0) =\mP(\xi U =0)$ for all $U \in \mathcal{U}$, where $\mathcal{U}$ is a separating class of $S$ which consists of relatively compact sets.
\item $\limsup_{n \rightarrow \infty} \mE \xi_n C  \leq \mE  \xi  C$ for all compact $C \in S$.
\end{enumerate}
\end{proposition}
Another criteria for convergence of random measures is the following
\begin{theorem}[Part of Theorem 16.16 of \cite{MR2161313}] \label{thm_Kallenberg}
Let $\{ \xi_n : n \in \mN \}$ be a sequence of random measures and let $\xi$ be a simple point process on a locally compact, second countable Hausdorff space $S$. Then $\xi_n \to_d \xi$ as $n \rightarrow \infty$ if and only if $\xi_n B \to_d \xi B$ as $n \rightarrow \infty$ for all $B \subseteq S$ which are relatively compact and satisfy $\xi \partial B = 0$ almost surely.
\end{theorem}

We recall some properties of the length $L_n$ of Kingman's $n$-coalescent tree.
\begin{theorem}[Theorem 2.3 of \cite{MoehlePitters2015}]
Let $L_n$ be the length of Kingman's coalescent tree of size $n \geq 2$. Then as $n \rightarrow \infty$, $L_n/2 - \log n$ converges almost surely to a standard-Gumbel distribution
\end{theorem}
For our purpose the following simple consequence will be important.
\begin{corollary} \label{cor_Ln}
The sequence of random variables $L_n/ \log n$ converges almost surely to $2$ as $n \rightarrow \infty$.
\end{corollary}

From~\eqref{Ln_exp} we immediately obtain
\begin{align}\label{lemma_ELn}
\mE L_n = 2 \sum_{k=1}^{n-1} \frac{1}{k}, \quad n \in \mN .
\end{align}
A direct consequence of this lemma is the following corollary.
\begin{corollary} \label{cor_ELn}
For $t_2 \geq 0$, $\mE L_{\lfloor n^{t_2} \rfloor} /(2 \log n)$ converges to $t_2$, as $n \rightarrow \infty$.
\end{corollary}

We also need the following crucial bound on the mean increments of $t\mapsto L_{\lfloor n^t\rfloor}$.
\begin{lemma}[Lemma 4 of \cite{Pitters2019}] \label{lemma_cdf}
Let $n \in \mN$ and $H_n^* := \sum_{k=2}^n 1/k$ for $n \geq 1$ and $H_0^* := 0$. Then the function
$$F_n(t) := \frac{H_{\lfloor n^{t} \rfloor-1}^*}{\log n}, \quad t \in [0,1]$$
is a distribution function on $[0,1]$ with $F_n(t) \leq t$ for all $t \in [0,1]$. Consequently, one has
$$\mu_n(A) \leq \lambda(A), \quad A \in \cB([0,1]),$$
where $\mu_n$ is the finite measure induced by $F_n$ and $\lambda$ is the Lebesgue measure.
\end{lemma}
A consequence of this lemma is
\begin{align*}
\frac{1}{\log n}(H_{\lfloor n^{t} \rfloor-1}^* - H_{\lfloor n^{s} \rfloor-1}^*) = \mu_n([s,t)) \leq t-s, \quad s,t \in [0,1], s \leq t
\end{align*}
and a slight generalisation (replacing $t \in [0,1]$ by $t \in [0,c]$ for an arbitrary $c>0$ in Lemma \ref{lemma_cdf}) even shows
\begin{align}\label{eq_harmonic}
\frac{1}{\log n}(H_{\lfloor n^{t} \rfloor-1}^* - H_{\lfloor n^{s} \rfloor-1}^*) \leq t-s, \quad s,t \in [0,\infty), s \leq t.
\end{align}

Now we restate the expectation and the variance of $S$ given in~\cite{Watterson1975}.
\begin{theorem}[\cite{Watterson1975}]\label{thm_pitters}
For $n \in \mN$ and $t \in [0,\infty)^2$ it holds
\begin{align*}
\mE S(n,t)&= t_1 H_{\lfloor n^{t_2} \rfloor-1}, \\
\Var S(n,t) &= t_1 H_{\lfloor n^{t_2} \rfloor-1} + t_1^2 H^{(2)}_{\lfloor n^{t_2} \rfloor-1},
\end{align*}
where $H^{(b)}_n := \sum_{k=1}^n 1/k^b$ for $n \in \mN$ and $b>0$ and $H_n := H_n^{(1)}$.
\end{theorem}
%With $\Pi_n^K((0,t]) = \tilde{S}(n,t) = S(n,(t_1/\log n, t_2))$ for $t \in [0,1]^2$ it follows immediately:
%\begin{corollary}\label{cor_pitters}
%For $n \in \mN$ and $t \in [0,1]^2$ it holds
%\begin{align*}
%\mE\Pi_n^S((0,t]) &= \frac{t_1}{\log n} H_{\lfloor n^{t_2} \rfloor-1}, \\
%\Var\Pi_n^S((0,t]) &= \frac{t_1}{\log n} H_{\lfloor n^{t_2} \rfloor-1} + \Big(\frac{t_1}{\log n}\Big)^2 H^{(2)}_{\lfloor n^{t_2} \rfloor-1},
%\end{align*}
%where $H^{(b)}_n := \sum_{k=1}^n 1/k^b$ for $n \in \mN$ and $b>0$ and $H_n := H_n^{(1)}$.
%\end{corollary}
We will also need the expectation of $K(n,t)$, which is
\begin{align}\label{expectaion_K}
\mE K(n,t) = 1+ \sum_{k=2}^{\lfloor n^{t_2} \rfloor} \frac{t_1}{t_1+k-1}.
\end{align}
This is a simple consequence of the representation
$$K(n,t) =_d 1+ \sum_{k=2}^{\lfloor n^{t_2} \rfloor} B_k(t_1),$$
where $(B_k(t_1))_{k \geq 2}$ is a sequence of Bernoulli random variables having success parameter $t_1 /(t_1+k-1)$, cf.~Theorem 3 in \cite{Pitters2019}. This representation is sometimes referred to as Feller's coupling.

\smallskip

Finally, we compute the variance of $S(n,t) - K(n,t)$.
\begin{lemma}\label{lemma_Var}
It holds that
$$\Var(S(n,t) - K(n,t)) = \sum_{k=2}^{\lfloor n^{t_2}\rfloor } \frac{t_1^2 \big(t_1^2 + 3 t_1(k-1) + (k-1)^2  \big)}{(k-1)^2(k-1+t_1)^2}.$$
\end{lemma}
\begin{proof}
  This is a direct calculation, see e.g.~the proof of Theorem 1 in~\cite{Pitters2019}.
\end{proof}

\section{Proofs}
Before we prove our convergence results we need to verify that $\Pi^S_n$ and $\Pi^K_n$ can be extended to random measures.
\begin{proof}[Proof of Proposition \ref{prop_measures}]
We apply a Caratheodory type extension for random measures (see e.g. Proposition 1.9.33 of Molchanov \cite{Mol_01}). For that, we have to extend $\Pi^S_n$ and $\Pi^K_n$ to a ring of subsets of $[0,\infty)^2$. Since $\Pi^S_n$ and $\Pi^K_n$ should be (almost surely) additive on this ring we have to do the following.
First, we generalize $\Pi^S_n$ and $\Pi^K_n$ to
$$\mathcal{Q}:= \Big\{[s_1,t_1) \times [s_2,t_2) : 0 \leq s_1 \leq t_1 \leq 1, 0 \leq s_2 \leq t_2 \leq 1\Big\} ,$$
which is the set of rectangles in $[0,\infty)^2$ excluding their upper and right boundary. For $[s,t) :=[s_1,t_1) \times [s_2,t_2) \in \mathcal{Q}$ we split up this set in the following way:
$$[s,t) = \Big([0,t)\setminus \big([0,s_1)\times [0,t_2) \cup [0,t_1) \times [0,s_2) \big) \Big) \cup [0,s).$$
Then we use the definition of $\Pi^S_n$ and $\Pi^K_n$ evaluated on the appearing sets on the right-hand side:
\begin{align*}
\Pi^S_n[s,t) := \Pi^S_n[0,t) + \Pi^S_n[0,s) - \big(\Pi^S_n([0,t_1) \times [0,s_2)) + \Pi^S_n([0,s_1) \times [0,t_2)) \big)
\end{align*}
and the analogue for $\Pi_n^K$:
\begin{align*}
\Pi^K_n[s,t) := \Pi^K_n[0,t)  + \Pi^K_n[0,s) - \big(\Pi^K_n([0,t_1) \times [0,s_2)) + \Pi^K_n([0,s_1) \times [0,t_2)) \big).
\end{align*}
Now consider
\begin{align}\label{set_R}
\mathcal{R} := \Big\{ \bigcup_{k=1}^n B_k : B_1,\ldots,B_n \in \mathcal{Q} \text{ disjoint }, n \in \mN \Big\},
\end{align}
which is the set of disjoint unions of rectangles in $[0,\infty)^2$. We still would like to have that $\Pi_n^S$ and $\Pi_n^K$ are additive (almost surely). So we need for $B = \bigcup_{k=1}^n B_k$ with $B_1,\ldots,B_n \in \mathcal{Q}$ disjoint:
$$\Pi^S_n B := \sum_{k=1}^n \Pi^S_n B_k \quad \text{and} \quad \Pi^K_n B := \sum_{k=1}^n \Pi^K_n B_k.$$
By this construction, $\Pi_n^S$ and $\Pi_n^K$ are $\sigma$-additive random functions on the ring $\mathcal{R}$ of subsets of $[0,\infty)$. Using the above mentioned version of Caratheodory's extension theorem for random measures, it follows that $\Pi^S_n$ and $\Pi_n^K$ can be extended from $\mathcal{R}$ to random measures on $\cB(\mathcal{R}) = \mathcal{B}([0,\infty)^2)$. 
\end{proof}

We go on with the following lemma which is needed to apply Proposition \ref{prop_Kallenberg} and Theorem \ref{thm_Kallenberg}.
\begin{lemma}
$\Pi^S_n$, $\Pi^K_n$ and $\Pi$ are simple point processes, i.e. point processes which fulfil $\Pi^S_n\{t \}, \Pi^K_n\{t \}, \Pi\{t \} \leq 1$ for $t=(t_1,t_2) \in [0,\infty)^2$.
\end{lemma}
\begin{proof}
We already know that the $\Pi$'s are random measures. Since their evaluations on all Borel sets $B \in \cB([0,\infty)^2)$ are obviously integer-valued they are point processes. It remains to show that they are simple, which means that the evaluation on singletons is less than $1$. But this follows by the fact that the Poisson process $\mathfrak{P}$ has almost surely no mass on points because its intensity measure is $1/2 \lambda^2$.
\end{proof}

%\subsection{Proof of Theorem \ref{main_result}}
\begin{theorem}\label{main_result}
The sequence of random measures $\{\Pi^S_n, n \in \mN\}$ converges weakly to $\Pi$ for $n \rightarrow \infty$, where $\Pi$ is a Poisson random measure on $[0,\infty)^2$ with intensity measure $\lambda^2$.
\end{theorem}
\begin{proof}[Proof of Theorem \ref{main_result}]
We apply Proposition \ref{prop_Kallenberg}. Hence, we have to show the following conditions:
\begin{enumerate}[(i)]
\item $\lim_{n \rightarrow \infty} \mP(\Pi^S_n U =0) =\mP(\Pi U =0)$ for all $U \in \mathcal{U}$, where $\mathcal{U}$ is a separating class of $[0,\infty)^2$ which consists of relatively compact sets.
\item $\limsup_{n \rightarrow \infty} \mE \big[ \Pi^S_n C \big] \leq \mE \big[ \Pi C \big]$ for all compact $C \in \mathcal{B}([0,\infty)^2)$.
\end{enumerate}
We start by showing condition (i). As a separating class of $[0,\infty)^2$ we use the set $\mathcal{R}$ defined in \eqref{set_R} which consists of disjoint unions of rectangles in $[0,\infty)^2$. Let be
$$U = \bigcup_{k=1}^m B_k  \in \mathcal{R}$$
with disjoint $B_k = [s^k_1,u^k_1) \times [s^k_2,u^k_2)$. We have to show that $\mP(\Pi^S_n U =0)$ converges to
\begin{align}\label{eq_111}
\mP(\Pi U =0) = \exp\big( - \sum_{k=1}^m (u^k_1-s^k_1)(u^k_2-s^k_2) \big).
\end{align}
The latter equality holds since $\Pi$ is a Poisson point process with intensity measure $\lambda^2$ and $\lambda^2(U) = \sum_{k=1}^m (u^k_1-s^k_1)(u^k_2-s^k_2)$. 

\smallskip

For fixed $n$ we have
\begin{align}\label{eq_112}
\mP(\Pi_n^S U =0) &=  \mP(\sum_{k=1}^m \Pi_n^K B_k  =0) \\
&=\mP( \Pi_n^S B_k  =0 \text{ for } k=1,\ldots,m).
\end{align}
Since the $B_k$ are disjoint and $\Pi^S_n$ is defined via the Poisson point process $\mathfrak{P}$, it follows that the $\Pi^S_n B_k$ are independent conditionally given
$$L =  \Big\{ L_{\lfloor n^{s^k_1} \rfloor},L_{\lfloor n^{u^k_1}\rfloor},L_{\lfloor n^{s^k_2} \rfloor},L_{\lfloor n^{u^k_2}\rfloor}, k=1,\ldots,m \Big\}.$$ More precisely it holds that
\begin{align}\label{help9}
\begin{split}
&\quad \mP( \Pi^S_n B_k  =0 \text{ for } k=1,\ldots,m) \\
&= \mE \Big[ \mP(\Pi^S_n B_k  =0 \text{ for } k=1,\ldots,m \,|\, L_{\lfloor n^{s^k_1} \rfloor},L_{\lfloor n^{u^k_1}\rfloor},L_{\lfloor n^{s^k_2} \rfloor},L_{\lfloor n^{u^k_2}\rfloor}, k=1,\ldots,m) \Big] \\
&=  \mE  \Big[ \prod_{k=1}^m \mP(\Pi^S_n B_k  =0 \,|\, L_{\lfloor n^{s^k_1} \rfloor},L_{\lfloor n^{u^k_1}\rfloor},L_{\lfloor n^{s^k_2} \rfloor},L_{\lfloor n^{u^k_2}\rfloor}, k=1,\ldots,m) \Big].
\end{split}
\end{align}
Again because of the definition $\Pi^S_n$ via the Poisson point process $\mathfrak{P}$ it holds
\begin{align*}
&\quad \mP(\Pi^S_n B_k  =0 \,|\, L_{\lfloor n^{s^k_1} \rfloor},L_{\lfloor n^{u^k_1}\rfloor},L_{\lfloor n^{s^k_2} \rfloor},L_{\lfloor n^{u^k_2}\rfloor}, k=1,\ldots,m)\\
&= \exp\big( -\frac{1}{2 \log n} (u^k_1-s^k_1)(L_{\lfloor n^{u^k_2}\rfloor}-L_{\lfloor n^{s^k_2} \rfloor}) \big).
\end{align*}
Thus, using Corollary \ref{cor_Ln} it follows
$$\lim_{n \rightarrow \infty} \mP(\Pi^S_n B_k  =0 \,|\, L_{\lfloor n^{s^k_1} \rfloor},L_{\lfloor n^{u^k_1}\rfloor},L_{\lfloor n^{s^k_2} \rfloor},L_{\lfloor n^{u^k_2}\rfloor}, k=1,\ldots,m) = \exp\big( -(u^k_1-s^k_1)(u^k_2-s^k_2) \big)$$
almost surely for $k=1,\ldots,m$. We combine this with dominated convergence in \eqref{help9} (possible because probabilities are bounded by $1$) to obtain that
\begin{align*}
\lim_{n \rightarrow \infty} \mP( \Pi^S_n B_k  =0 \text{ for } k=1,\ldots,m) &= \prod_{k=1}^m \exp\big( -(u^k_1-s^k_1)(u^k_2-s^k_2) \big)\\
&= \exp\big( -\sum_{k=1}^m (u^k_1-s^k_1)(u^k_2-s^k_2) \big),
\end{align*}
which together with \eqref{eq_111} and \eqref{eq_112} shows condition (i).

\smallskip

Now we show condition (ii) via proving the stronger result
$$\lim_{n \rightarrow \infty} \mE \Pi^S_n C = \mE \Pi C  =\lambda^2(C),$$
where the last identity holds because $\Pi$ is a Poisson point process with intensity measure $\lambda^2$. First let be $C = [0,t) = [0,t_1) \times [0,t_2)$ (although this set is not compact). As we mentioned above it holds $\tilde{S}(n,t) =_d \Pi_n[0,t)$ and hence,
\begin{align}\label{help7}
\mE \Pi_n^S [0,t)  = \mE  \tilde{S}(n,t) =  \mE S(n,t_1/(\log n),t_2) .
\end{align}
Combining \eqref{help7} with Theorem \ref{thm_pitters} we get 
\begin{align}\label{help 8}
\lim_{n \rightarrow \infty} \mE \Pi^S_n [0,t) =  \lim_{n \rightarrow \infty} \frac{t_1}{\log n} \sum_{k=1}^{\lfloor n^{t_2} \rfloor-1} 1/k =t_1 t_2 = \lambda([0,t)),
\end{align}
which shows the case $C = [0,t)$. 

\smallskip

For a set $B = [s,u) = [s_1,u_1) \times [s_2,u_2)$ we will show $\lim_{n \rightarrow \infty} \mE \Pi^S_n B  = \mE \Pi B  =(u_1-s_1)(u_2-s_2)$. Let us begin as follows
\begin{align}\label{help1}
\begin{split}
\mE  \Pi^S_n B  &= \mE \Big[ (\mE  \Pi^S_n B)  | L_{\lfloor n^{u_2} \rfloor} ,L_{\lfloor n^{s_2} \rfloor}  \Big] \\
&= \frac{u_1-s_1}{2 \log n} \mE \big[ L_{\lfloor n^{u_2} \rfloor} -L_{\lfloor n^{s_2} \rfloor} \big].
\end{split}
\end{align}
By \eqref{Ln_exp} it follows
\begin{align}\label{help2}
\begin{split}
\mE \big[ L_{\lfloor n^{u_2} \rfloor} -L_{\lfloor n^{s_2} \rfloor} \big] &=  \sum_{k=\lfloor n^{s_2} \rfloor+1}^{\lfloor n^{u_2} \rfloor} \mE  \tau_k  \\
&= \sum_{k=2}^{\lfloor n^{u_2} \rfloor} \mE \tau_k  -  \sum_{k=2}^{\lfloor n^{s_2} \rfloor} \mE  \tau_k  \\
&= \mE  L_{\lfloor n^{u_2} \rfloor} - \mE  L_{\lfloor n^{s_2} \rfloor}.
\end{split}
\end{align}
From Corollary \ref{cor_ELn} we know that  $\mE  L_{\lfloor n^{t} \rfloor} $ behaves like $2t \log n$ for large $n$ we obtain
\begin{align}\label{help6}
\begin{split}
\lim_{n \rightarrow \infty }\mE  \Pi^S_n B &= \lim_{n \rightarrow \infty } \Big(\frac{u_1-s_1}{2 \log n} 2(u_2-s_2) \log n\Big)\\
&= (u_1-s_1)(u_2-s_2).
\end{split}
\end{align}
Now let us assume that $C \in \mathcal{B}([0,\infty)^2)$ is compact. To show (i), we write it as a (countable) union of sets which have almost the same form as $B$. More concretely, let
$$C = \bigcup_{k=1}^\infty \overline{B_k},$$
with $B_k =  [s^k,u^k) = [s_1^k,u_1^k) \times [s_2^k,u_2^k) \subset [0,\infty)^2$ such that $B_k \cap B_j = \emptyset$ for $k \neq j$.

\smallskip

\underline{The case $\inf_{k \in \mN} s_2^k >0$}: First we assume $\inf_{k \in \mN} s_2^k >0$, i.e. the set $C$ does not touch the $x$-axes. Since $\Pi^S_n$ has no mass on one-dimensional sets and is a measure almost surely we get with a Fubini flip
\begin{align}
\begin{split}
\mE \Pi^S_n C  &= \mE \Big[ \Pi^S_n \big( \bigcup_{k=1}^\infty \overline{B_k} \big) \Big] = \mE \Big[ \Pi^S_n \big( \bigcup_{k=1}^\infty B_k \big) \Big] = \sum_{k=1}^\infty \mE \Pi^S_n B_k .
\end{split}
\end{align}
To calculate $\lim_{n \rightarrow \infty} \sum_{k=1}^\infty \mE  \Pi^S_n B_k $ we will switch the limit and the sum via dominated convergence. For that it is necessary to show that there is a sequence $a_k$ such that $\mE \Pi^S_n B_k \leq a_k$ for all $k \in \mN$ and $\sum_{k=1}^\infty a_k < \infty$. With \eqref{help1} and \eqref{help2} we see
\begin{align}\label{help3}
\mE  \Pi^S_n B_k  = \frac{u^k_1-s^k_1}{2 \log n} \big(\mE  L_{\lfloor n^{u^k_2} \rfloor}  - \mE L_{\lfloor n^{s^k_2} \rfloor} \big).
\end{align}
We have to go in more detail concerning the expectations on the right-hand side. By Lemma \ref{lemma_ELn} it holds
\begin{align*}
\mE L_{\lfloor n^{u^k_2} \rfloor}  - \mE L_{\lfloor n^{s^k_2} \rfloor}  &=  2 \Big( \sum_{l=1}^{\lfloor n^{u^k_2}\rfloor-1} \frac{1}{l}  - \sum_{l=1}^{\lfloor n^{s^k_2} \rfloor-1} \frac{1}{l} \Big).
\end{align*}
Because of our assumption $\inf_{k \in \mN} s_2^k >0$ we can assume that $n$ is large enough for $\lfloor n^{s^k_2} \rfloor \geq 3$ for all $k \in \mN$. Hence, the first summand in both sums cancels, i.e.
\begin{align}\label{help4}
\begin{split}
\mE  L_{\lfloor n^{u^k_2} \rfloor}- \mE L_{\lfloor n^{s^k_2} \rfloor}  &=  2 \Big( \sum_{l=2}^{\lfloor n^{u^k_2} \rfloor-1} \frac{1}{l}  - \sum_{l=2}^{\lfloor n^{s^k_2} \rfloor-1} \frac{1}{l} \Big) \\
&=2 \big( H^*_{\lfloor n^{u^k_2} \rfloor-1} - H^*_{\lfloor n^{s^k_2} \rfloor-1} \big),
\end{split}
\end{align}
where $H^*_n = \sum_{k=2}^{n} 1/l$ for $n \in \mN$. 
%From Lemma \ref{lemma_cdf} we know that for fixed $n \in \mN$ with $n \geq 2$,
%$$F(t) = \frac{H^*_{\lfloor n^{t} \rfloor-1}}{ \log n}, \quad t \in [0,1],$$
%is the distribution function corresponding to a measure $\mu_n$ on $[0,1]$ with $\mu_n(A) \leq \lambda(A)$.
From \eqref{eq_harmonic} we know
\begin{align}\label{help5}
\frac{1}{\log n} \big( H^*_{\lfloor n^{u^k_2} \rfloor-1} - H^*_{\lfloor n^{s^k_2} \rfloor-1} \big) \leq u^k_2 - s^k_2.
\end{align}
Combining \eqref{help3}, \eqref{help4} and \eqref{help5}, we obtain
\begin{align}
\mE \Pi^S_n B_k  \leq  \frac{1}{2}(u^k_1-s^k_1)2(u^k_2-s^k_2) = \lambda^2(B_k) =: a_k.
\end{align}
Furthermore, we have
$$\sum_{k=1}^\infty a_k =   \sum_{k=1}^\infty \lambda^2(B_k) =   \lambda^2\big(\bigcup_{k=1}^\infty B_k \big) =   \lambda^2(K) < \infty,$$
because the $B_k$'s are disjoint. So we can apply dominated convergence and get with \eqref{help6}
\begin{align}\begin{split}
\lim_{n \rightarrow \infty}\mE \Pi^S_n C  &=  \sum_{k=1}^\infty \lim_{n \rightarrow \infty}\mE  \Pi^S_n B_k  \\
&= \sum_{k=1}^\infty (u^k_1-s^k_1)(u^k_2-s^k_2) \\
&= \lambda\big(\bigcup_{k=1}^\infty B_k \big) \\
&= \lambda\big(\bigcup_{k=1}^\infty \overline{B_k} \big) \\
&=  \lambda(C).
\end{split}
\end{align}
\underline{The general case}: Now let be $C \in \cB([0,\infty))$ be an arbitrary compact Borel set. We define 
$$\overline{s}_1 \coloneqq \sup\{s_1 \geq 0 : (s_1,s_2) \in C \} \quad \text{and} \quad \overline{s}_2 \coloneqq \sup\{s_2 \geq 0 : (s_1,s_2) \in C \},$$ 
which are both finite since $C$ is compact. For an arbitrary but fixed $\eps>0$ we define $O^\eps := C \cap [0,\overline{s}_1) \times [0,\eps)$ and $M^\eps := C  \cap [0,\overline{s}_1) \times [\eps,\overline{s}_2)$ such that obviously $C = \overline{O^\eps} \cup \overline{M^\eps}$ and $O^\eps \cap M^\eps = \emptyset$. The set $M^\eps$ is a set of the form of the first case (when the set does not touch the $x$-axes), hence,
$$\lim_{n \rightarrow \infty} \mE \Pi^S_n \overline{M^\eps} = \lim_{n \rightarrow \infty} \mE \Pi^S_n M^\eps  = \lambda(M^\eps) \in \Big[\lambda(C)-\overline{s}_1\eps, \lambda(C)\Big].$$
For the set $O^\eps$ we see with the calculations for sets of the form $[0,t), t \in [0,\infty)^2$:
$$\lim_{n \rightarrow \infty} \mE \Pi^S_n \overline{O^\eps} = \lim_{n \rightarrow \infty} \mE \Pi^S_n O^\eps  \leq \lim_{n \rightarrow \infty} \mE \Pi^S_n([0,\overline{s}_1) \times [0,\eps)) = \overline{s}_1\eps.$$
By these two statements we obtain
\begin{align}
\lim_{n \rightarrow \infty} \mE \Pi^S_n C =  \lim_{n \rightarrow \infty} \mE \Pi^S_n O^\eps  +\lim_{n \rightarrow \infty} \mE  \Pi^S_n M^\eps  \in \Big[\lambda(C)-\overline{s}_1\eps, \lambda(C)+\overline{s}_1\eps\Big].
\end{align}
Since $\eps>0$ was arbitrary we obtain $\lim_{n \rightarrow \infty} \mE \Pi^S_n C  =\lambda(C)$.
\end{proof}

\subsection{Proof of Theorem \ref{main_result2}}
We start with a lemma which gives a first idea of the convergence of $\Pi_n^K$. First, we note that $\Pi_n^S  - \Pi_n^K$ is a random measure for fixed $n \in \mN$. Indeed, by definition it holds $\Pi_n^S(B)  - \Pi_n^K(B) \geq 0$ almost surely for all $B \in \cB([0,\infty)^2)$. Furthermore, $\Pi_n^S$ and $\Pi_n^K$ are both $\sigma$-additive which implies the same property for $\Pi_n^S  - \Pi_n^K$.
\begin{lemma}\label{lemma_help1}
The sequence of random measures $\{ \Pi_n^S - \Pi_n^K; n \in \mN \}$ converges weakly to $0$.
\end{lemma}
\begin{proof}
For $t \in [0,\infty)^2$ and $n \in \mN$ we define $\Delta(n,t) := \Pi_n^S[0,t) - \Pi_n^K[0,t)$. First, we show that the random variable $\Delta(n,t)$ converges to $0$ in $L^2$ for $n \rightarrow \infty$. Because $\mE\Delta(n,t)^2 = (\mE\Delta(n,t))^2 + \Var(\Delta(n,t))$ it is sufficient to show that $ \mE \Delta(n,t)$ and $\Var(\Delta(n,t))$ converge to $0$. Since $\Pi_n^S[0,t) =_d \tilde{S}(n,t)$ and $\Pi_n^K[0,t) =_d \tilde{K}(n,t)-1$ we can apply Theorem \ref{thm_pitters} and \eqref{expectaion_K} to obtain
\begin{align*}
0 \leq \mE \Delta(n,t) &= \mE \tilde{S}(n,t) -  (\mE \tilde{K}(n,t) -1) \\
&= \frac{t_1}{\log n} \sum_{k=2}^{\lfloor n^{t_2} \rfloor} \frac{1}{k} - (1+ \sum_{k=2}^{\lfloor n^{t_2} \rfloor}  \frac{t_1/ \log n}{k + t_1 / \log n}) +1 \\
&= \frac{t_1}{\log n} \sum_{k=2}^{\lfloor n^{t_2} \rfloor} \Big(\frac{1}{k} - \frac{1}{k + t_1 / \log n} \Big) \\
&= \frac{t_1}{\log n} \sum_{k=2}^{\lfloor n^{t_2} \rfloor} \frac{t_1 / \log n}{k^2 + kt_1 / \log n } \\
&\leq \Big(  \frac{t_1}{\log n}\Big)^2  \sum_{k=2}^{\lfloor n^{t_2} \rfloor} \frac{1}{k^2} \rightarrow 0
\end{align*}
for $n \rightarrow \infty$. To show convergence of the variance we note that for $n$ big enough and a constant $C>0$ it holds:
\begin{align*}
\Var(\Delta(n,t)) &= \Var(\Pi_n^S[0,t) - \Pi_n^K[0,t))\\
&= \Var(\tilde{S}(n,t) - \tilde{K}(n,t))\\
&= \sum_{k=2}^{\lfloor n^{t_2}\rfloor } \frac{(t_1/\log n)^2 \big((t_1/\log n)^2 + 3 t_1/\log n(k-1) + (k-1)^2  \big)}{(k-1)^2(k-1+t_1/\log n)^2}\\
&\leq C \Big(\frac{t_1}{\log n} \Big)^2  \sum_{k=2}^{\lfloor n^{t_2}\rfloor } \frac{1}{(k-1)^2} \rightarrow 0.
\end{align*}
Here, the third identity is from Lemma \ref{lemma_Var}.

\smallskip

It follows that $\Delta(n,t) = \Pi_n^S[0,t) - \Pi_n^K[0,t)$ converges to $0$ in distribution for all $t \geq 0$ and since $\Pi_n^S  - \Pi_n^K$ is almost surely a measure we get
$$ \Pi_n^S(B) - \Pi_n^K(B) \to_d 0$$
for $n \rightarrow \infty$ for all $B \in \cB([0,\infty)^2)$. Theorem \ref{thm_Kallenberg} then implies that $\{ \Pi_n^S  - \Pi_n^K; n \in \mN \}$ converges weakly to the zero measure on $[0,\infty)$.
\end{proof}

\begin{proof}[Proof of Theorem~\ref{main_result2}]
We will directly apply the results of Theorem \ref{main_result} and Lemma \ref{lemma_help1}. The remarks on the top of page 25 in Billingsley \cite{Billingsley1968} and Lemma \ref{lemma_help1} imply that $\{ \Pi_n^S   - \Pi_n^K; n \in \mN \}$ converges to $0$ also in probability. Applying Theorem 4.4 of Billingsley \cite{Billingsley1968} and Theorem \ref{main_result} we obtain 
$$(\Pi_n^S,\Pi_n^S  - \Pi_n^K) \to_d (\Pi,0).$$
In the last step we use the continuous mapping theorem (see e.g. Corollary 1 on page 31 in Billingsley \cite{Billingsley1968}) with the function $f(x,y) = (x,x-y)$ to obtain the statement of our result:
\begin{align*}
(\Pi_n^S,\Pi_n^K) = f(\Pi_n^S,\Pi_n^S  - \Pi_n^K) \to_d f(\Pi,0) = (\Pi,\Pi).
\end{align*}
\end{proof}

\bibliographystyle{alpha}
\bibliography{literature}

\begin{thebibliography}{{Pit}19}

\bibitem[Bil68]{Billingsley1968}
Patrick Billingsley.
\newblock {\em Convergence of probability measures}.
\newblock John Wiley \& Sons, Inc., New York-London-Sydney, 1968.

\bibitem[Dur08]{Durrett2008}
Richard Durrett.
\newblock {\em {Probability models for DNA sequence evolution}}.
\newblock 2008.

\bibitem[Eth11]{Etheridge2011}
Alison Etheridge.
\newblock {\em {Some mathematical models from population genetics}}, volume
  2012 of {\em {Lecture Notes in Mathematics}}.
\newblock Springer, Heidelberg, 2011.
\newblock Lectures from the 39th Probability Summer School held in Saint-Flour,
  2009.

\bibitem[Eva07]{Evans2007}
Steven~N Evans.
\newblock {\em Probability and Real Trees: {\'E}cole D'{\'E}t{\'e} de
  Probabilit{\'e}s de Saint-Flour XXXV-2005}.
\newblock Springer, 2007.

\bibitem[Kal05]{MR2161313}
Olav Kallenberg.
\newblock {\em {Probabilistic symmetries and invariance principles}}.
\newblock {Probability and its Applications (New York)}. Springer, New York,
  2005.

\bibitem[Kim69]{Kimura1969}
Motoo Kimura.
\newblock {The number of heterozygous nucleotide sites maintained in a finite
  population due to steady flux of mutations}.
\newblock {\em Genetics}, 61(4):893--903, 1969.

\bibitem[Mol17]{Mol_01}
I.~Molchanov.
\newblock {\em Theory of Random Sets}.
\newblock Springer London, 2 edition, 2017.

\bibitem[MP15]{MoehlePitters2015}
M.~M\"{o}hle and H.~Pitters.
\newblock Absorption time and tree length of the {K}ingman coalescent and the
  {G}umbel distribution.
\newblock {\em Markov Process. Related Fields}, 21(2):317--338, 2015.

\bibitem[{Pit}19]{Pitters2019}
Helmut {Pitters}.
\newblock {The number of cycles in a random permutation and the number of
  segregating sites jointly converge to the Brownian sheet}.
\newblock {\em arXiv e-prints}, page arXiv:1903.04906, Mar 2019.

\bibitem[Tav04]{Tavare2004}
Simon Tavar{\'e}.
\newblock {Ancestral inference in population genetics}.
\newblock In {\em {Lectures on probability theory and statistics}}, volume 1837
  of {\em {Lecture Notes in Math.}}, pages 1--188. Springer, Berlin, 2004.

\bibitem[Wat75]{Watterson1975}
G.A. Watterson.
\newblock {On the number of segregating sites in genetical models without
  recombination}.
\newblock {\em Theoretical Population Biology}, 7(2):256--276, 1975.

\end{thebibliography}
\end{document}